\documentclass{amsart}
\usepackage{amsmath,amsthm,amssymb, cancel}

\theoremstyle{plain}
\newtheorem{theorem}{Theorem}[section]
\newtheorem{lemma}[theorem]{Lemma}
\newtheorem{corollary}[theorem]{Corollary}

\theoremstyle{definition}
\newtheorem{definition}[theorem]{Definition}
\newtheorem{example}[theorem]{Example}

\theoremstyle{remark}

\title{A short nonstandard proof of \\ the Radon-Nikodym theorem}

\author{Takashi Matsunaga}
\address{Department of Medical Informatics, Osaka International Cancer Institute, Osaka City, Japan}
\email{matsunaga-ta@nifty.com}

\begin{document}

\renewcommand{\thefootnote}{\fnsymbol{footnote}}

\begin{abstract} 
Using nonstandard analysis, an intuitive and very short proof of the Radon-Nikodym theorem is provided.
\end{abstract}

\maketitle

\section{Introduction}

Although the Radon-Nikodym theorem is one of the most fundamental theorems in  measure theory, proofs in standard mathematics are long and not straightforward ({\it e.g.} Bruckner \cite{Bruckner}).  

On the other hand, Luxemburg \cite{Luxemburg} offered a nonstandard proof of the Radon-Nikodym theorem using a result of Riesz, while Ross \cite{Ross} gave another nonstandard proof of the theorem through a conditional expectation argument. Recently, Ma \cite{Ma} presented a Loeb space version of the theorem.     

Here we provide an intuitive and very short proof of the Radon-Nikodym theorem using nonstandard analysis.

\section{Preliminaries}
We assume that the reader is familiar with basic nonstandard analysis. A self-contained quick introduction to nonstandard analysis is presented in the Appendix.  Davis \cite{Davis} is a standard textbook for nonstandard analysis. We write $x \simeq y$, if $x$ is infinitesimaly close to $y$.

Throughout this note, we fix a (standard) finite measure space $(M, \mathcal{M}, \nu)$. Let  $\mu$ be any (standard) finite signed measure on $(M, \mathcal{M})$.

\begin{lemma}
If there exists a measurable $P_0 \in \mathcal{M}$ with $\mu(P_0)>0$ , then there exists a measurable $P \in \mathcal{M}$ with $\mu(P)>0$ such that for any $Q \in \mathcal{M}$  and $Q \subseteq P$, $ \mu(Q) \geq 0 $.
\end{lemma}
\begin{proof}
Suppose that there exists a nonempty $P_1 \in \mathcal{M}$ with $\mu(P_1) < 0$ and $P_1 \subseteq P_0$. Define for each $n \in \mathbb{N^+}$, $P_{n+1} \in \mathcal{M}$ inductively by
$$ P_{n+1}  \subseteq P_0 - \bigcup_{k =1}^n P_k \ {\rm and} \ \mu(P_{n+1}) \le -\frac{1}{n_0} \ ({\rm with \ the \ least \ natural \ number} \ n_0).$$
$P = P_0- \bigcup_{k=1}^\infty P_k$ has the desired property.
\end{proof}

\begin{corollary}{\rm (the Hahn decomposition)}{\label{Hahn}}
There exists a measurable $M^+ \in \mathcal{M}$ such that  $\mu(M^+) = \max_{S \in \mathcal{M}}{\mu(S)}$,  and for any $T \in \mathcal{M}$ with $\mu(T) = \mu (M^+)$, $ \nu(M^+) \geq \nu(T)$.
\end{corollary}

\begin{lemma}\label{concurrence}
There exists a *-finite algebra $\mathcal{M}_0 \subseteq \mathcal{^*M}$ such that for any $S \in \mathcal{M}$, ${^*S} \in \mathcal{M}_0$. 
\end{lemma}
\begin{proof}
Apply the Concurrence Principle.
\end{proof}

\begin{corollary}
There exsists a measurable $M_0^+ \in \mathcal{M}_0$ such that  ${^*\mu}(M_0^+) = \max_{S \in \mathcal{M}_0}{^*\mu}(S)$,  and for any $T \in \mathcal{M}_0$ with ${^*\mu}(T) = {^*\mu}(M_0^+)$, ${^*\nu}(M_0^+) \geq {^*\nu(T)}$.
\end{corollary}

\begin{corollary}\label{key}
$^*\mu(M_0^+ \ominus {^*(M^+)}) = 0, \  ^*\nu(M_0^+ \ominus {^*(M^+)}) = 0$. 
 \end{corollary}

\begin{lemma}\label{absolute}
If $\mu$ is non-negative and absolute continuous with respect to $\nu$, then for any $E \in \mathcal{^*M}$ with $^*\nu(E) \simeq 0$, $^*\mu(E) \simeq 0$.    
\end{lemma}
\begin{proof}
If otherwise, we have $E \in \mathcal{^*M}$ with $^*\nu(E) \simeq 0$ such that $ ^*\mu(E)  \not\simeq 0$. By the Transfer Principle, for each $k \in \mathbb{N^+}$ we obtain $F_k \in \mathcal{M}$ such that $\nu(F_k) <1/2^k$ and $\mu(F_k) \ge {\rm st}(\, ^*\mu(E)) >0$, where ${\rm st}(x)$ is the standart part of $x$. Let $F = \bigcap_{n=1}^{\infty} \bigcup_{k=n}^{\infty} F_k$. Using the Borel-Cantelli lemma, it is a routine to see that $\nu(F) =0$ and $\mu(F) >0$, which cotradicts the absolute continuity.
\end{proof}

\section{A nonstandard proof the Radon-Nikodym theorem}

In this section, we fix a (standard) finite signed measure $\lambda$ on $(M, \mathcal{M})$, which is absolutely continuous with respect  to $\nu$. We can assume that $\lambda$ is non-negative thanks to the Hahn decomposition (Corollary \ref{Hahn}).

\begin{lemma}\label{main1}
There exists a non-negative $f_0: {^*M} \rightarrow { ^* \mathbb R}$ such that for any $S \in \mathcal{M}$,
$$  {^*\lambda}({^*S}) = {\int}_{^*S} f_0(x) d{^*\nu}, \ \ \ \ \  {\int}_{f_0(x) \ge K} f_0(x) d{^*\!\nu} \simeq 0 \ \  (K: {\rm infinite}). $$ 
 \end{lemma}
\begin{proof}
To prove the first equality, for each atom $A \in \mathcal{M}_0$, if  $^*\nu(A)>0$, then define $f_0(x)  (x \in A)$ by
$ f_0(x) ={^*\lambda(A)}/{^*\nu(A)},$
otherwise set $f_0(x)=0  (x \in A)$. For the second equality, if otherwise, there exists an infinite $K_0$ such that
  $$ 0 \le {\int}_{f_0(x) \ge K_0} f_0(x) d{^*\!\nu}  \le {^*\!\lambda}(^*M) , $$ 
which means ${^*\nu}(\{x \in {^*M} | f_0 (x) \ge K_0\}) \simeq 0$ because ${^*\lambda}(\,^*\!M)$ is finite by assumption. By Lemma \ref{absolute}, the absolute continuity of $\mu$ with respect to $\nu$ yields
  $$ {\int}_{f_0(x) \ge K_0} f_0(x) d{^*\!\nu} = {^*\mu}(\{x \in {^*M} | f_0 (x) \ge K_0\}) \simeq 0. $$
\end{proof}

\begin{lemma}
For each $a \in \mathbb{R}$, let $F_0(a) = \{ x \in {^*M} | f_0(x) \ge {^*a} \} \in \mathcal{M}_0$. Then there exists a measurable set $F(a) \in \mathcal{M}$ such that
 $^*\nu(F_0(a) \ominus {^*F(a)}) = 0$
\end{lemma}
\begin{proof}
Apply  Corollary \ref{key} for $ \mu = \lambda - a\nu$. It follows that $F_0(a) = M_0^+$ and $^*F(a) = {^*(M)^+}$.
\end{proof}

\begin{corollary}\label{main2}
For each $a < b \in \mathbb{R}$, let $F_0(a, b) = \{ x \in {^*M} | {^*a \le f_0(x)} < {^*b} \} \in \mathcal{M}_0$. Then there exists a measurable set $F(a, b) \in \mathcal{M}$ such that
 $^*\nu(F_0(a, b) \ominus {^*F(a, b)}) = 0$
 \end{corollary}

\begin{theorem}{\rm (the Radon-Nikodym Theorem)}
There exists a non-negative measurable function $f: M \rightarrow \mathbb{R}$ such that for any $S \in \mathcal{M}$,
$$  \lambda(S) = \int_S f(x) d\nu. $$ 
\end{theorem}
\begin{proof}
We use the notations of  Lemma \ref{main1} and Corollary \ref{main2}. Denote the characteristic function of  $S (S \in \mathcal{M})$ by $\mathbb J _S$. For $a<b \in \mathbb{R}$, let $\mathbb{I}(a,b) = \mathbb{J}_{F(a,b)}$.
For each $n \in \mathbb{N}$ define standard $f_n(x)$ and $ f(x)$  by
$$ f_n(x) = \sum_{k=0}^{n2^n-1} \frac{k}{2^n}\mathbb I(\frac{k}{2^n}, \frac{k+1}{2^n})  \uparrow f(x) \ \ \ (n \uparrow \infty). $$
By construction, $f_0(x) \ge {^*f_n(x)}$ and if $f_0(x) < n$ then
$$ 0 \le f_0(x) - {^*f_n(x)} \le \frac{1}{2^n}, $$
with the help of Corollary \ref{main2}.
Thus for any $n \in \mathbb{N^+}$
$$  {\int}_{^*M} |f_0(x)-{^*f_n(x)}|  d{^*\!\nu} \le  {\int}_{f_0(x) \ge n} f_0(x)  d{^*\!\nu} + \frac{1}{2^n}{^*\mu(^*M)}.  $$
Using the Overflow Principle, we obtain for some infinite $K \in \mathbb{^*N^+}$
$$  {\int}_{^*M} |f_0(x)-{^*f_K(x)}|  d{^*\!\nu} \le  {\int}_{f_0(x) \ge K} f_0(x)  d{^*\!\nu} + \frac{1}{2^K}{^*\mu(^*M)} \simeq 0, $$
by Lemma \ref{main1}. The monotone convergence theorem yields
$$  {\int}_{^*M} |^*f_K(x)-{^*f(x)}|  d{^*\!\nu} \simeq 0. $$
Hence we obtain 
$$  {\int}_{^*M} |f_0(x)-{^*f(x)}|  d{^*\!\nu} \simeq 0,  $$
which completes the proof.
\end{proof}

\section{Concluding Remarks}
To the author's knowledge, this is a shortest proof  of the Radon-Nikodym theorem for those familiar with nonstandard analyis.

\subsection*{Disclosure statement and funding}
There are no interests to declare. No funding was received.

\section{Appendix: a quick introduction to nonstandard analysis }
In this appendix, we assume that the reader is fimiliar with basics of (naive) set theory and first order logic.

Nonstandard analysis is a theory founded by Abraham Robinson in the 1960s, motivated largely by the revival of Leibnizian infinitesimals.  In essence, he constructed a nonstandard extension $^*U$ (Theorem \ref{extension}) of a universe (Definition \ref{universe}) containing $\mathbb{R}$. $^*U$ contains a proper extension of $\mathbb{R}$ denoted by $^*\mathbb{R}$, which is logically similar to $\mathbb{R}$ but includes ideal elements such as infinitesimals and infinite numbers.

We need several definitions and lemmas to prove Thoerem \ref{extension}.

\begin{definition}\label{universe}(Universe)
A universe $U$ is a set satisfying the following conditions:
\begin{enumerate}
\item $u \in v$ and $v \in U$ imply $u \in U$, 
\item $u \in U$ and $v \in U$ imply $\{u, v\} \in U$,
\item $u \in U$ implies $\bigcup u \in U$,
\item $u \in U$ implies $\mathcal{P}(u) \in U$, where $\mathcal{P}(u)$ is the power set of $u$. 
\end{enumerate}
\end{definition}

\begin{definition}(Formula without Constants)
A formula without constants is defined by the folloing rules only: 
\begin{enumerate}
\item For the variables $x$ and $y$, $x=y$ and $x \in y$ are (atomic) formulae without constants,
\item If $\phi$ is a formula without constants, $\lnot\phi$, $\exists x \phi$, and $\forall x \phi$ are formulae without constants, 
\item If $\phi$ and $\psi$ are formulae without constants, $\phi \land \psi$, $\phi \lor \psi$, and $\phi \to \psi \equiv \lnot\phi \lor \psi $ are formlae without constants, 
\end{enumerate}
A bounded variable $x$ in $\phi$ is a variable that appears in the form of  $\exists x \phi$ or $\forall x \phi$. Other variables in $\phi$ are free variables. 

The free variables in a formula $\phi$ is often indicated explicitly by $\phi(x_1, x_2, \cdots, x_n)$.
For a universe $U$, a formula without constants $\phi(x_1, x_2, \cdots, x_n)$, and the constants $c_1, c_2, \cdots, c_n \in U$, we can verify whether $\phi(c_1, c_2, \cdots c_n)$ holds or not in $U$ by letting the bounded variables in $\phi$ range over $U$. 

As usual, $\exists x \in a [\phi] \equiv \exists x [x \in a \land \phi]$, $\forall x \in a [\phi] \equiv  \forall x [x \in a \to \phi]$, and $\exists !x [\phi(x)] \equiv \exists x [\phi(x) \land \forall y [\phi(y) \to y=x]]$.
\end{definition}

\begin{definition}(Internal Formula without Constants)
For a formula without constants $\phi$, the internal formula without constants $^*\phi$ is given by replacing all the $\in$'s in $\phi$ with the binary relation $^*\!\!\in$ on $^*U$ (defined in Theorem  \ref{extension}). As in the previous definition, for the constants $c_1, c_2, \cdots, c_n \in {^*U}$, $^*\phi$ holds or not in $^*U$.
\end{definition}

\begin{example}(Universe containing $\mathbb{R}$)
Let $V_0 = \mathbb{R}$ and define $V_{n+1}$ by $V_{n+1} = \bigcup V_n$ inductively. Set $U_0 = \bigcup V_n$. Define $U_{n+1}$ by $U_{n+1} = U_n \cup \mathcal{P}(U_n)$ inductively. Set $U = \bigcup U_n$. It is straightforward to verify that $U$ is a universe.
\end{example}
Note that any universe containing $\mathbb{R}$ has various mathematical objects for analysis. Usually we adopt a universe $U$ that has all relevant mathematical objects.

\begin{definition}(Filter Basis, Filter, Ultrafilter)
A filter basis $\mathcal{F}$ on a set $I$ is a subset of $\mathcal{P}(I)$ satisfiying (1) and (2). A filter $\mathcal{F}$ is a filter basis satisfying (3). An ultrafilter $\mathcal{F}$ is a filter satisfying (4). 
\begin{enumerate}
\item $\phi \not\in \mathcal{F}$ and $I \in \mathcal{F}$, 
\item If $A \in \mathcal{F}$ and $B \in \mathcal{F}$, then $A \cap B \in \mathcal{F}$,
\item If  $A \in \mathcal{F}$ and $A \subseteq B$, then $B \in \mathcal{F}$,
\item If $\mathcal{G}$ is a filter and $\mathcal{F} \subseteq \mathcal{G}$, then $\mathcal{F} = \mathcal{G}$, that is, $\mathcal{F}$ is maximal under $\subseteq$.
\end{enumerate}
\end{definition}

\begin{lemma}\label{ultraexist}
For a filter basis $\mathcal{F}_0$ on $I$, there exists an ultrafilter $\mathcal{F}$ containing $\mathcal{F}_0$.
\end{lemma}
\begin{proof}
Let $\mathcal{F}_1 = \{ X \in \mathcal{P}(I) | X \supseteq A \ {\rm for \ some} \ A \in \mathcal{F}_0 \}$. It is easy to verify that $\mathcal{F}_1$ is a filter. If $\mathcal{F}_1$ is not maximal, there exsits a filter $\mathcal{F}_2 \supsetneq \mathcal{F}_1$. If $\mathcal{F}_2$ is not maximal, there exsits a filter $\mathcal{F}_3 \supsetneq \mathcal{F}_2$ and so on. Finally we have a maximal $\mathcal{F}$. More formally, the existence of $\mathcal{F}$ follows by Zorn's lemma. 
\end{proof}

\begin{lemma}\label{either}
For an ultrafilter $\mathcal{F}$ on $I$ and $A \subseteq I$, either $A \in \mathcal{F}$ or $I - A \in \mathcal{F}$ holds.
\end{lemma}
\begin{proof}
Suppose that $A \not\in \mathcal{F}$. Let $\mathcal{G}  = \{X \in \mathcal{P}(I) | X \cup A \in \mathcal{F} \}$. It is a routine to check that $\mathcal{G}$ is a filter containing $\mathcal{F}$. Hence $\mathcal{G} = \mathcal{F}$ by hypothesis. Obviously $I-A \in \mathcal{G}$. This completes the proof. 
\end{proof}

\begin{definition}(Ultrapower)
Let $V$ be an infinite set and $I$ be an infinite index set. Let denote the set of all maps from $I$ to $V$ by $V^I$. Let $<a(i)>, <b(i)> \in V^I$. Define the equivalence relation on $V^I$ by $\{i \in I | a(i) =b(i) \} \in \mathcal{F}$ (this is well-defined if $\mathcal{F}$ is an ultrafilter on $I$). The ultrapower of $V$ over an ultrafilter $\mathcal{F}$ on $I$ is the set of equivalence classes of $V^I$.
\end{definition}

\begin{theorem}\label{extension}(Nonstandard Extension)
For a given universe $U$, there exsits a nonstandard extension $^*U$ and a binary relation $^* \! \! \in$ on $^*U$ that satisfy the following two conditions for any formula without constants $\phi$:
\begin{enumerate}
\item the Transfer Principle: There exists an injective map $^*: U \rightarrow {^*U}$ such that $\phi(c_1, c_2, \cdots, c_n)$ holds in $U$ if only if (iff) $^*\phi(^*c_1, ^*c_2, \cdots, ^*c_n)$ holds in $^*U$, where $c_1, c_2, \cdots, c_n \in U$ and $^*c_i$'s are the images of $c_i$'s under the map $^*$.
\item the Concurrence Principle: If $\phi(x, y, c_1, c_2, \cdots c_n)$ ($c_1, c_2, \cdots c_n \in U$) is "concurrent" with respect to $X (\subseteq U)$, then there exists an element $y \in {^*U}$ such that $^*\phi(\,^*x, y, ^*c_1, ^*c_2, \cdots, ^*c_n)$ holds in $^*U$ for all $x \in X$. Here "conccurent" with respect to $X$ means that for any finite collection $x_i \in X$, there exists an element $y \in U$ such that $\phi(x_i, y, c_1, c_2, \cdots, c_n)$ holds in $U$ for all $x_i$.
\end{enumerate}
\end{theorem}
\begin{proof}
The construction of $^*U$ proceeds as follows. Consider the index set $I$ of all finite subsets of $U$. For each $i \in I$, let $\mu(i) = \{ j \in I \mid i \subseteq j \}$. Since $\mu(i_1) \cap \mu(i_2) \cap \cdots \cap \mu(i_n) = \mu(i_1 \cup i_2 \cup \cdots \cup i_n)$, this family is a filter basis on $I$ so that  there exists an ultrafilter $\mathcal{F}$ on $I$ containing all the $\mu(i)$'s by Lemma \ref{ultraexist}.  Denote the ultrapower of $U$ over $\mathcal{F}$ by $^*U$ and the equivalence classes by $[a(i)], [b(i)]$. The binary relation $[a(i)] \, {^*\!\!\in}\ [b(i)]$ is also well-defined by $\{ i \in I \mid a(i) \in b(i) \} \in \mathcal{F}$. The map * is defined by $a \mapsto [a(i) = a]$. This map is injective by definition.

The Transefer Principle is the special case of the next lemma. For the Concurrence Principle, by assumption, there exists an element $y(i) \in U$ such that $\phi(x, y(i), c_1, c_2, \cdots, c_n)$ holds for all $x \in i \cap X$ so that we obtain the desired result again using the next lemma.
\end{proof}

\begin{lemma}(\cancel{L}os's Theorem)
Under the same situation as in the proof of Theorem \ref{extension}, for any formula without constants $\phi(x,y,\dots,z)$, $^*\phi([a(i)], [b(i)], \cdots, [c(i)])$ holds in $^*U$ iff $\{i \in I \ |\ \phi(a(i), b(i), \cdots ,c(i)) {\rm \ holds \ in} \ U \} \in \mathcal{F}$.
\end{lemma}
\begin{proof}
We apply mathematical induction on the number of logical symbols in $\phi(x,y, \dots, z)$. If there is no logical symbols in $\phi(x,y,\cdots, z)$, $\phi(x,y,\cdots, z)$ is $x\in y$ or $x=y$ so that the conclusion follows by definiton. Using De-Morgan's law, it suffice to prove the cases  of $\exists, \land$ and $\lnot$. Let us consider the case of $\exists$. Clearly $\exists z \phi([a(i)],[b(i)], \cdots, z)$ holds in $^*U$ iff for some $[c(i)]$ $\phi([a(i)],[b(i)], \cdots, [c(i)])$ holds in $^*U$. By the hypothesis of induction, this occurs iff $\{i \in I \ |\ \phi(a(i), b(i), \cdots, c(i)) \\ {\rm \ holds \ in} \ U \} \in \mathcal{F}$. The definition of a filter yields the case of $\land$. The case of $\lnot$ comes from Lemma \ref{either}.
\end{proof}

\begin{definition}(Standard, Internal)
An element $u \in U$ and the correponding $^*u \in {^*U}$ are called standard. An element $v \in {^*U}$ is called internal.
\end{definition} 

\begin{example}(Transfer Principle I: Embedding, *-Omission, *-Pair, Exstension)
\begin{enumerate}
\item If $A \in U$ and $a \in A$, then $a \in U$ by the transitivity of $U$. Thus $^*a$ is defined and $a \in A$ iff $^*a  {\,^*\!\in} {\,^*\!A}$ by the Transfer Principle so that we usually leave out $^*$ from $^*\!\in$ for simplicity. Since the map $^*: U \mapsto \,^*U$ is injective, $A$ is embedded into $^*A$ and we often assume $A \subseteq {^*A}$.
\item Since $\forall x, y \exists !z \forall u [u \in z \to u=x \lor u=y]$ holds in $U$, 
$\forall x, y \exists !z \forall u [u \in z \to u=x \lor u=y]$ holds in $^*U$ by the Transfer Principle. Hence we can define the *-pair denoted by $^*\{x, y\}$. By defitition $z=\{x,y\}$ iff $^*z= {^*\{}{^*x}, {^*y}\}$. Similarly we can also define the *-ordered pair ${^*\!<}x, y>$ so that $z=<x,y>$ iff $^*z= {^*\!<}{^*x}, {^*y}>$. Thus $<x,y> \in f$ iff ${^*\!<}{^*x}, {^*y}> \in {^*f}$.
\item For $a, b \in \mathbb{R}$ $a<b$ iff ${^*a} {\ ^*\!\!<} {\ ^*b}$ by (2). If we assume $\mathbb{R} \subseteq {^*\mathbb{R}}$ as in (1), $^*\!\!<$ is regarded as an extension of $<$ so that $^*$ is often omitted,
\item For $A, B \in U$ and $f: A \mapsto B$, $b=f(a)$ iff $^*b= {^*f(^*a)}$ by (2). If we assume $A \subseteq {^*A}$ and $B \subseteq {^*B}$ as in (1), $^*f$ is again an extension of $f$ so that $^*$ is often omitted.
\end{enumerate}
\end{example}

\begin{example}(Concurrence Principle I)
Since $\phi(x,y,\mathbb{R}) \equiv x<y \land x \in \mathbb{R}\land y \in \mathbb{R}$ is concurrent with respect to $\mathbb{R}$, we obtain $y \in {^*\mathbb{R}}$ such that for any $x \in \mathbb{R}$ $^*x < y$. That is, $y$ is an infinite number and $1/y$ is an infinitesimal. For $x, y \in {^*\mathbb{R}}$, we write  $x \simeq y$ if $x-y$ is infinitesimal.
\end{example}

\begin{example}(Standard Part)
If $c \in {^*\mathbb{R}}$ is finite, $\sup(\{ x \in \mathbb{R} \,|\, ^*x  <  c \} )$ is called the standard part of $c$ and denoted by ${\rm st}(c)$. It is easy to check $c \simeq \,^*\!({\rm st}(c))$.
\end{example}

\begin{lemma}(Overflow Principle)
Let $^*\phi(n, x_1, x_2, \cdots, x_n)$ be an internal formula without constants and suppose that for the constants $c_1, c_2, \cdots, c_n \in {^*U}$ $^*\phi(^*n, c_1, c_2, \cdots, c_n)$ holds in $^*U$ for any $n \in \mathbb{N}$. Then for some infinite $N \in {^*\mathbb{N}}$ $^*\phi(N, c_1, c_2, \cdots, c_n)$ holds in $^*U$.
\end{lemma}
\begin{proof}
By the principle of mathematical induction,
$$\forall x_1 \cdots \forall x_n \exists n_0 \in \mathbb{N} \ [\lnot \phi(n_0, x_1, \cdots, x_n) \land [\forall n \in \mathbb{N} \ \lnot \phi(n, x_1, \cdots, x_n) \to n_0 \le n]]$$
holds in $U$. By the Transfer Principle, 
$$\forall x_1 \cdots \forall x_n \exists n_0 \in {^*\mathbb{N}} \ [\lnot \,^*\!\phi(n_0, x_1, \cdots, x_n) \land [\forall n \in {^*\mathbb{N}} \ \lnot \,^*\!\phi(n, x_1, \cdots, x_n) \to n_0 \le n]]$$
also holds in $^*U$. Hence substituting $c_1, c_2, \cdots, c_n$ for $x_1, x_2, \cdots, x_n$, 
$$\exists n_0 \in {^*\mathbb{N}} \ [\lnot \,^*\!\phi(n_0, c_1, \cdots, c_n) \land [\forall n \in {^*\mathbb{N}} \ \lnot \,^*\!\phi(n, c_1, \cdots, c_n) \to n_0 \le n]]$$
holds in $^*U$. That is, $n_0$ is the least number such that $^*\phi(n_0, c_1, \cdots, c_n)$ does not hold in $^*U$.  Since $n_0$ is infinite by hypothesis, $N = n_0-1$ has the desired property.
\end{proof}

\begin{lemma}(Convergence)\label{convergence}
Let $\{a_n\}$ be a real sequence.Then $\lim_{n \rightarrow \infty} a_n = a$ iff $^*a_N \simeq \,^*a$ for all infinite $N \in \,^*\mathbb{N}$. 
\end{lemma} 
\begin{proof}
Suupose that $\lim_{n \rightarrow \infty} a_n = a$. Then by definition, for any  
$\epsilon \in \mathbb{R}_{>0}$ there exists a natural number $n_0 \in \mathbb{N}$ such that  $\forall n \in \mathbb{N} (n>n_0 \to |a_n-a|<\epsilon)$ holds. By the Transfer Principle, $\forall n \in \, {^*\mathbb{N}} (n>{^*n_0} \to |\,^*a_n-\,^*a|<\,^*\epsilon)$ holds. Thus for all infinite $N$, obviously $N \ge{^*n_0}$ so that $|\,^*a_N - \,^*a| < \,^*\epsilon$. Since $\epsilon \in \mathbb{R}_{>0}$ is arbitary, $^*a_N \simeq \,^*a$. Conversely suppose that $^*a_N \simeq \,^*a$ for any inifinte $N$.  Then for any $\epsilon \in \mathbb{R}_{>0}$ $|\,^*a_N-\,^*a|<\,^*\epsilon$. Thus cleary for some (actually any) inifinite $n_0$, $n>n_0$ implies $|\,^*a_n-\,^*a|<\,^*\epsilon$. That is, $\exists n_0 \in \,^*\mathbb{N} \ \forall n \in \,^*\mathbb{N} \ (n>n_0 \to |\,^*a_n-\,^*a| <\,^*\epsilon)$ holds. By the Transfer Principle, $\exists n_0 \in \mathbb{N} \ \forall n \in \mathbb{N}\ (n > n_0 \to |a_n-a| <\epsilon)$ holds. Since $\epsilon \in \mathbb{R}_{>0}$ is arbitary, we have the conclusion.
\end{proof}

\begin{definition}(Transfer Principle II: *-Property, *-Finitenss, *-Finite Sum)
\begin{enumerate}
\item Let $P (\in U)$ define some property $Prop$. We say $u$ is $Prop$ if $u \in P$. In this situation, we say $v$ is *-$Prop$ if $v \in \,^*P$. An example is the following.  For $A \in U$, if $P \equiv \mathcal{P}_F(A)$ (the set of all the finite subsets of $A$), then $u (\in P)$ is a finite subset of $A$ and $v (\in \,^*P)$ is a *-finite subset of $^*A$.
\item Denote the finite sum of the elements of a finite subset of $\mathbb{R}$ by $\Sigma: \mathcal{P}_F(\mathbb{R}) \mapsto \mathbb{R}$. Then, we obtain the *-finite sum $^*\Sigma: \,^*\mathcal{P}_F(\mathbb{R}) \mapsto \ ^*\mathbb{R}$. That is, *-finite sum is defined on all the *-finite subset of $^*\mathbb{R}$.
\end{enumerate}
\end{definition}
 
\begin{example}(Concurrence  Principle II)
$\phi(x,y,\mathcal{P}_F([0,1]) ) \equiv x \in y \land y \in \mathcal{P}_F([0,1])$ is concurrent with respect to $[0,1]$. Hence there exists an element $y \in \,^*\mathcal{P}_F([0,1])$ such that $^*x \in y$ for all $x \in [0,1]$. In other words, there exists a *-finite subset of $^*[0,1]$ that contains all the elements of $[0,1]$, if we assume $[0,1] \subseteq \,^*[0,1]$.
\end{example}

\begin{lemma}(Uniform Continuity)\label{uniformcontinuity} 
Let $f(x)$ be a function on $[0,1]$. Then  $f(x)$ is uniformly contiuous on $[0,1]$ iff $\forall x,y \,^*\!\in {^*[0,1]} \  (x \simeq y \to {^*f(x)} \simeq {^*f(y)})$ holds.
\end{lemma}
\begin{proof}
The proof is very similar to that of Thoerm \ref{convergence}.
\end{proof}

\begin{definition}(Good *-Partition of $[0,1]$) 
$p \equiv \{0=a_0<a_1<\cdots<a_i<\cdots<a_n =1\} \ (a_i \in \mathbb{R},n \in \mathbb{N})$ is a partition of $[0,1]$. By the Transfer Principle, $P\equiv \{0=a_0<a_1<\cdots<a_i<\cdots<a_N =1\} \ (a_i \in {^*\mathbb{R}}, N \in {^*\mathbb{N}})$
 is a *-partition of $^*[0,1]$. $P$ is called "good", in case for any $x \in [0,1]$, $^*x \in P$. The term "good" is used only in the next example. 
\end{definition}

\begin{example}(Riemann-Stieltjes Integral)
Let $f(x): [0,1] \mapsto \mathbb{R}$ be a continuous function and $g(x): [0,1] \mapsto \mathbb{R}$ be a non-decreasing function. For $p \equiv \{0=a_0<a_1<\cdots<a_n =1\}$ (a partition of $[0,1]$), set $ S(p) \equiv \sum_{k=1}^n f(a_{k-1})(g(a_k)-g(a_{k-1}))$. $S$ is a function from the set of all the partition of [0,1] $I$ to $\mathbb{R}$ so that by the Transfer Principle, $^*S: {^*I} \mapsto {^*\mathbb{R}}$. In other words, for a *-partition of $^*[0,1], P \equiv \{0=a_0 < a_1< \cdots <a_N=1\}$ $^*S(P) = {^*\sum}_{k=1}^N f(a_{k-1})(g(a_k) -g(a_{k-1}))$, where $N \in {^*\mathbb{N}}$ and $^*\sum$ is the *-finite sum. Note that *'s are omitted from $^*f, ^*g$ and $^*-$. Suppose that $P$ and $P'$ are "good" *-partitions of $^*[0,1]$. Then $^*S(P) \simeq {^*S(P')}$. To see this, let $P''$ be the combined *-partition of $P$ and $P'$. Then, it is a straightfoward to verify that $S(P) \simeq S(P'')$ and $S(P') \simeq S(P'')$ by using Lemma \ref{uniformcontinuity}.  
\end{example}

\end{document}